\documentclass{amsart}
\newtheorem{theorem}{Theorem}[section]
\newtheorem{corollary}{Corollary}[section]

\newtheorem{proposition}{Proposition}[section]
\newtheorem{lemma}{Lemma}[section]
\theoremstyle{definition}
\newtheorem{definition}{Definition}[section]

\newcommand{\la}{\left\langle}
\newcommand{\ra}{\right\rangle}
\newcommand{\lp}{\left(}
\newcommand{\rp}{\right)}

\newcommand{\eival}{\varepsilon}

\newcommand{\nto}{\stackrel{n}{\longrightarrow}}
\newcommand{\cD}{\mathcal{D}}
\newcommand{\cDz}{{\cD_0}}

\newcommand{\cR}{\mathcal{R}}

\newcommand{\cH}{\mathcal{H}}
\newcommand{\C}{\mathbb{C}}
\newcommand{\R}{\mathbb{R}}

\newcommand{\N}{\mathbb{N}}
\newcommand{\Nz}{{\N_0}}

\newcommand{\txi}{\tilde{\xi}}
\newcommand{\teta}{\tilde{\eta}}
\newcommand{\tf}{\tilde{f}}

\newcommand{\heta}{\hat{\eta}}
\newcommand{\hxi}{\hat{\xi}}

\newcommand{\Ilam}{{I_\lambda}}

\newcommand{\Clam}{C^\lambda}
\newcommand{\Clamof}[1]{C_{\lambda,#1}}
\newcommand{\Klam}{{K_\lambda}}
\newcommand{\etalam}{\eta_\lambda}
\newcommand{\Wlam}{W_\lambda}
\newcommand{\taulam}{{\tau_\lambda}}

\newcommand{\ttau}{\tilde{\tau}}
\newcommand{\ttaulam}{\ttau_\lambda}

\newcommand{\tauz}{\tau_\emptyset}

\newcommand{\pilamof}[1]{\pi_{\lambda,#1}}

\newcommand{\taueta}{\tau_\eta}
\newcommand{\tauxi}{\tau_\xi}
\newcommand{\Hlamof}[1]{H_{\lambda,#1}}

\newcommand{\hpsi}{\hat{\psi}}

\newcommand{\alxiet}{\alpha_{\xi,\eta} }
\newcommand{\beetxi}{\beta_{\eta,\xi} }
\newcommand{\Axiet}{A_{\xi,\eta}}
\newcommand{\Betxi}{B_{\eta,\xi} }
\newcommand{\axiet}{a_{\xi,\eta}}
\newcommand{\betxi}{b_{\eta,\xi} }

\newcommand{\bg}{\bar{g}}

\newcommand{\rL}{\operatorname{\mathrm{L}}}
\newcommand{\Ann}{\operatorname{Ann}}

\newcommand{\WR}{\operatorname{Wr}}
\DeclareMathOperator*{\bchi}{\operatorname{\chi}}
\DeclareMathOperator*{\lspan}{\operatorname{span}}
\DeclareMathOperator*{\Ran}{\operatorname{Ran}}
\DeclareMathOperator*{\Dom}{\operatorname{Dom}}

\newcommand{\WRof}[1]{\WR\!\lp#1\rp}
\newcommand{\chiof}[1]{\bchi\!\lp#1\rp}

\newcommand{\lami}[1]{\lambda^{(#1)}}

\newcommand{\ACloc}{\operatorname{AC}_{\mathrm{loc}}}
\newcommand{\Cinf}{\mathcal{C}^{\infty}}

\title{Spectral Theory of Exceptional Hermite Polynomials}
\author{David G\'omez-Ullate}
\address{Escuela Superior de Ingenier\'ia, U.  C\'adiz, 11519 Puerto Real,
  Spain, \and
  Departamento de F\'isica Te\'orica, U. Complutense, 28040 Madrid, Spain}
\email{david.gomezullate@uca.es}

\author{ Yves Grandati}
\address{ L. Physique et Chimie Th\'eoriques, U. de
 Lorraine, 57078 Metz, Cedex 3, France}
\email{yves.grandati@univ-lorraine.fr}

\author{Robert Milson}
\address{Dept. of Mathematics and Statistics, Dalhousie U.,
  Halifax NS, Canada B3H 3J5} \email{rmilson@dal.ca}

\begin{document}
\maketitle

\begin{abstract}
  In this paper we revisit exceptional Hermite polynomials from the
  point of view of spectral theory, following the work initiated by
  Lance Littlejohn. Adapting a result of Deift, we provide an
  alternative proof of the completeness of these polynomial
  families. In addition, using equivalence of Hermite Wronskians we
  characterize the possible gap sets for the class of exceptional
  Hermite polynomials.
\end{abstract}
\section{Introduction}

Consider a Sturm-Liouville problem (SLP) on
$(-\infty,\infty)$:
\begin{equation}
  \label{eq:WRSL}
 - ( W y')' - R y = \eival W y,
\end{equation}
where $W,R = O(e^{-x^2})$ as $|x| \to \infty$.  Remarkably, there
exist a large class of such eigenvalue problems with polynomial
eigenfunctions.  Classical Hermite polynomials, which correspond to
the case of $W=e^{-x^2}, R=0$, are just one, very particular, example.
Families of polynomials that arise as eigenfunctions of such SLP are
called exceptional Hermite orthogonal polynomials.

% The corresponding spectral problem is limit-point at $\pm \infty$ with
% a discrete spectrum that is bounded below and accumulates at $\infty$
% \cite[Section ??]{CL}.  Since $\eta$ is even, the corresponding
% eigenfunctions have a well-defined parity.
% Remarkably, there exist a large class of $\eta\in \cR$ for which
% \eqref{eq:WetaSL} constitutes an eigenvalue problem with polynomial
% solutions.  Classical Hermite polynomials, which correspond to the
% case of $\eta(x) = 1$, are just one example.  Families of polynomials
% that arise as eigenfunctions of \eqref{eq:WetaSL} are called
% exceptional Hermite orthogonal polynomials.

Being the eigenfunctions of an SLP, exceptional Hermite polynomials
are complete families of orthogonal polynomials with weight
$W(x)>0,\; x\in (-\infty,\infty)$ \cite{GKM09,GGM14}. However, unlike
the classical Hermite polynomials, the degree sequence of an
exceptional family has a finite number of gaps in the degree sequence;
there is a finite number of so-called exceptional degrees for which
there is no corresponding eigenpolynomial.

The well-known Bochner theorem \cite{Bo29} asserts that if
$y_n(x),\; n=0,1,2,\ldots$ with $\deg y_n =n$ is a family of
polynomials that satisfies a second-order eigenvalue equation
\[ p y_n'' + q y_n' + r y_n = \eival_n y,\] then necessarily
$p(x),q(x)$ are polynomials with $\deg p\le 2, \deg q \le 1$ and $r$
is a constant.  Multiplying \eqref{eq:WRSL} by $-W^{-1}$ gives
\begin{equation}
  \label{eq:tauWR}
  y''+q y' + r y   = -\eival y
\end{equation}
where
\[ q = W^{-1} W', \quad r = W^{-1} R. \] In order for \eqref{eq:tauWR}
to have infinitely many polynomial eigenfunctions, it is necessary for
$q(x),r(x)$ to be rational functions with $\deg q\le 1$ and
$\deg r \le 0$.  To obtain non-classical polynomials, it is necessary
for $q$ and $r$ to have poles.  Bochner's theorem then implies that
such a non-classical family of polynomial eigenfunctions must have
gaps in the degree sequence.

If we assume that there is at most a finite number of such gaps, then
it is possible to show (see \cite{GFGM19,GGM14}) that, necessarily,
\eqref{eq:tauWR} takes the form
\begin{equation}
  \label{eq:tauform}
  y''-2\lp x+\frac{\eta'}{\eta}\rp y'+ \lp \frac{\eta''}{\eta}+ 2 x
  \frac{\eta'}{\eta}\rp  y = \eival y,
\end{equation}
where $\eta(x)$ is a real polynomial without any real zeros .  Not
every choice of $\eta$ results in an eigenvalue relation
\eqref{eq:tauform} with polynomial eigenfunctions.  When it does,
however, one can show that there are precisely $\deg \eta$
``exceptional'' degrees.

% If $\deg \eta >1$, then the differential expression in
% \eqref{eq:tauform} has rational coefficients, and therefore does not
% satisfy the hypothesis of Bochner's theorem.  It follows that a family
% of polynomials that arise as solutions of \eqref{eq:WetaSL} with
% $\deg \eta >1$ must have gaps in the degree sequence.  Indeed, as we
% show below,  families of exceptional  Hermite polynomials are missing
% precisely $\deg \eta$ ``exceptional'' degrees.

Some examples of exceptional polynomials were investigated back in the
early 90s \cite{DEK94} but their systematic study started about 10
years ago, where a full classification was given for codimension one
\cite{GKM09}. The role of Darboux transformations in the construction
process was quickly recognized \cite{GKM10,Qu09,STZ10} and the next
conceptual step involved the generation of exceptional families by
multiple-step or higher order Darboux transformations \cite{GKM12}.

Exceptional polynomials appear in mathematical physics as bound states
of exactly solvable rational extensions \cite{GGM14,OS09,Qu09} and
exact solutions to Dirac's equation \cite{SH14}. They appear also in
connection with super-integrable systems \cite{PTV12,MQ13} and
finite-gap potentials \cite{HV10}.  From a mathematical point of view,
the main results are concerned with the full classification of
exceptional polynomials \cite{GFGM19}, properties of their zeros
\cite{GMM13,Ho15,KM15}, and recurrence relations \cite{Du15,GKKM16,
  KM20}.

Lance Littlejohn and collaborators wrote a series of papers analyzing
the spectral-theoretic properties of 1-step exceptional operators
\cite{LL15,LL16}.  It is the ambition of the present work to extend
this type of analysis to the class of multi-step exceptional Hermite
operators.  There are two primary results. First, we merge the
approach of Littlejohn with spectral-theoretic characterization of
Darboux transformations obtained by Deift \cite{deift} (see
\cite{GP96} for a further extension to the double-commutator method)
to provide a novel demonstration of the completeness of the
exceptional polynomials. Second, we characterize the possible gap sets
of families of  exceptional Hermite polynomial  in terms the
corresponding partition.

% Exceptional polynomials appear in a number of applications
% in mathematical physics, mostly as solutions to exactly solvable
% quantum mechanical problems describing bound states
% \cite{GGM14,Odake2009a,Quesne2009}. They appear also in
% connection with super-integrable systems
% \cite{Post2012,Marquette2013a}, exact solutions to Dirac's equation
% \cite{Schulze-Halberg2014}, diffusion equations and random processes
% \cite{Ho2011b}, finite-gap potentials \cite{Hemery2010} and point
% vortex models \cite{Kudryashov}.
% % And indeed, every family of exceptional Hermite polynomials
% % is missing precisely $\deg \eta$ ``exceptional'' degrees.

\section{Some spectral theory}
Let $\cR\subset \R[x]$ be the set of real-valued polynomials that have
no real zeros.  Consider a SLP
on
$(-\infty,\infty)$ of the form
\begin{equation}
  \label{eq:WetaSL}
 - ( W_\eta y')' - R_\eta y = \eival W_\eta y , \qquad \eta \in \cR,
\end{equation}
where
\begin{align}
  \label{eq:Weta}
  W_\eta
  &:= \eta^{-2}e^{-x^2},\\
  \label{eq:Reta}
  R_\eta
  &:= \eta^{-3}(\eta''+ 2 x\eta')e^{-x^2}
\end{align}
In this section we will consider Darboux transformations of 
self-adjoint operators corresponding to SLP belonging to the class
shown in \eqref{eq:WetaSL}.

Fix $\xi,\eta\in \cR$ and set
\begin{align}
  \label{eq:teta}
  \teta&:=  \eta e^{x^2},\\
  \label{eq:heta}
  \heta&:= \eta^{-1} e^{-\frac{x^2}2} ,
\end{align}
with $\txi, \hxi$ defined analogously.

Next, introduce the bilinear differential expression
\begin{equation}
  \label{eq:chider}
  \chiof{f,g}
  = f g''-2f'g'+f''g- 2x (fg'-f'g).
\end{equation}
and define the following first, and second order
differential expressions
\begin{align}
  \label{eq:alphadef}
  \alxiet y
  & =\eta^{-1}\WR(\xi,y)\\
%    =\frac{\xi}{\eta} \lp   y' -    \frac{\xi'}{\xi} y\rp  ,\\
  \label{eq:betadef}
  \beetxi 
  &= \alpha_{\teta,\txi}\\
  %   =\frac{\eta}{\xi}    \lp y' - \left (\frac{\mu'}{\mu} +
  % \frac{\eta'}{\eta}\rp y\rp,\\ 
  \label{eq:taudef}
  \taueta y
  &= \eta^{-1} \chiof{\eta,y}.
\end{align}
with $\tauxi$ defined analogously.  Note that $\taueta$ is
precisely the second-order differential expression in the left-side of
\eqref{eq:tauform}.

A direct calculation shows that
\begin{equation}
  \label{eq:albefadj}
  (\alxiet f) g W_\xi +  f
  (\beetxi g) W_\eta=  (\heta f\,  \hxi g)',
\end{equation}
with $W_\eta,W_\xi$ as defined in \eqref{eq:Weta}.  Consequently,
$\alxiet$ and $-\beetxi$ are formally adjoint with respect to weights
$W_\xi, W_\eta$.  To be precise, if $I\subset \R$ is a compact
interval, and $f,g$ sufficiently smooth functions defined on $I$, we
have
\begin{equation}
  \label{eq:albefadj2}
  \int_I  (\alxiet f) \bg\, W_\xi +
  \int_I  f  (\beetxi \bg), W_\eta  = \heta f\, \hxi \bg \Big|_I 
\end{equation}

Let us also recall Lagrange's identity
\[ (\taueta f) g \, W_\eta -  f(\taueta g)\, W_\eta  =
  W_\eta \WRof{f,g}.\]
Consequently, $\taueta$ is symmetric with respect to $W_\eta$ in the
sense that
\begin{equation}
  \label{eq:tausym}
  \int_I  (\taueta f) \bg W_\eta -
  \int_I  f (\taueta \bg)\, W_\eta  =  W_\eta \WRof{f,\bg} \Big|_I,
\end{equation}
An analogous relation holds for $\tauxi$ and $W_\xi$.

Let $(\cH_\eta,\la \cdot,\cdot \ra_\eta)=\rL^2(\R,W_\eta)$ be the
Hilbert space of square-integrable complex-valued functions defined on
$(-\infty,\infty)$, with $\cH_\xi$ defined analogously. Let
$\cDz$ denote the vector space of smooth, functions
with compact support, and let
\begin{equation}
  \label{eq:cDq}
  \cD_\xi := \{ f\in  \Cinf(\R) \colon  \WRof{f,\xi}(x)
  \equiv 0 \text{ for  $|x|$  sufficiently large} \}.
\end{equation}
In other words, an element of $\cD_\xi$ behaves like a multiple of
$\xi(x)$ outside a compact interval, but with the constant of
proportionality for large $x$ not necessarily equal to the constant of
proportionality for small $x$.  It is well known that $\cD_0$ is dense
in $\cH_\eta$ and $\cH_\xi$. Since elements of $\cD_\xi$ are square
integrable with respect to $W_\eta$ and since $\cD_0 \subset \cD_\xi$,
the latter is also a dense subspace.
 
Let $\axiet\colon \cH_\eta\to \cH_\xi,\;\betxi\colon
\cH_\xi \to \cH_\eta$
denote  densely defined first-order operators with 
action $\alxiet, \beetxi$, respectively, and with
\[ \Dom\axiet = \cD_\xi,\quad \Dom\betxi = \cD_0.\] Since
$\alxiet\xi\equiv0$, it follows that $\Ran\axiet = \cD_0$.  For
same reason and by \eqref{eq:albefadj2}, we have
$\Ran\betxi \subset \Ann_\eta\xi$, where
\begin{equation}
  \label{eq:Anndef}
 \Ann_\eta\xi = \{ f\in \cH_\eta \colon \la f, \xi \ra_\eta = 0.\}.
\end{equation}
Next, we
strengthen these assertions as follows.

\begin{proposition}
  \label{prop:RanA0}
  We have $\overline{\Ran \axiet} = \cH_\xi$.
\end{proposition}
\begin{proof}
  Since $\cD_0$ is dense in $\cH_\xi$, it suffices to show that
  $\cD_0 \subset \Ran \axiet$.  Let $h\in \cD_0$ be given.  Set
  \[ f(x) = \xi(x) \int_{-\infty}^x h\eta\xi^{-2},\quad x\in\R.\]
  Since $h\in \cD_0$, for $x$ sufficiently small, we have $f(x)=0$.
  For $x$ sufficiently large, we have $f(x) = C \xi(x)$ where
  $C=\int_\R h\eta\xi^{-2}.$ Hence, $f\in \cD_\xi$.  By
  \eqref{eq:alphadef} we have,
  \[ \alxiet f = \eta^{-1} \xi^2 \lp \xi^{-1} f\rp '=    h.\]
\end{proof}

\begin{proposition}
  \label{prop:qperp}
  We have $\overline{\Ran \betxi} =  \Ann_\eta\xi$.
\end{proposition}
\begin{proof}
  Let $\{ f_n\in \cD_0\}_{n\in \N}$ be a sequence such that
  $\beetxi f_n\nto h$ for some $h\in \cH_\eta$. Since
  $\alxiet\xi=0$, by \eqref{eq:albefadj}, we have
  \[\la \beetxi f_n,\xi \ra_\eta=0,\quad n\in \N.\]  Moreover, since
  \[\la h,\xi\ra_\eta =\la h-\beetxi f_n,\xi\ra_\eta\nto 0,\] we have
  $h\in \Ann_\eta\xi$.  Hence,
  $\overline{\Ran \betxi}\subseteq\Ann_\eta\xi$.
  
  We now prove the converse.  Let $h\in \Ann_\eta\xi$ be given.  Let
  $\{ h_n\in \cD_0\}_{n\in \N}$ be a sequence such that
  $\Vert h_n- h\Vert \nto 0$.  Choose a
  $p\in \cD_0$ such that $\la p,\xi\ra \ne 0$, and set
  \[ g_n = h_n - \frac{\la h_n,\xi\ra_\eta}{\la p,\xi\ra_\eta} p\in
    \cD_0,\quad n\in \N.\] By construction,
  $g_n \in\Ann_\eta\xi,\; n\in \N$. By assumption,
  $\la h_n,\xi\ra_\eta \nto 0$.  Hence,
  \[ \Vert g_n -h \Vert_\eta \leq \Vert h_n -h \Vert_\eta + \left|\frac{ \la
      h_n,\xi\ra_\eta}{\la
      p,\xi\ra_\eta }\right|\Vert p\Vert_\eta \nto 0.\] For $n\in \N$, set
\[ f_n(x) = \teta(x) \int_{-\infty}^x g_n \txi \teta^{-2}
  = \teta(x) \int_{-\infty}^x g_n \xi W_\eta ,\quad
  x\in \R.\] Since $g_n$ has compact support, $f_n(x) = 0$ for $x$
sufficiently small.  Since $g_n\in \Ann_\eta\xi$, we also have
$f_n(x) =0$ for $x$ sufficiently large. Hence, $f_n\in \cD_0$.
Observe that
  \[ \beetxi f_n = \txi^{-1} \teta^2 (\teta^{-1}f_n)'= g_n.\] Hence
  $g_n\in \Ran\betxi,\; n\in \N$.  We already showed that
  $g_n \nto h$. This proves that
  \[\Ann_\eta\xi\subseteq\overline{\Ran \betxi} .\]
\end{proof}

Let $\Axiet,\Betxi$ denote the maximal extensions of $\axiet,\betxi$,
respectively. Formally,
\begin{equation}
  \label{eq:domAdomB}
  \begin{aligned}
    \Dom\Axiet &= \{ f\in \cH_\eta \colon f\in \ACloc(\R); \alxiet f \in
    \cH_\xi\},\\
    \Dom\Betxi &= \{ g\in \cH_\xi \colon g\in \ACloc(\R); \beetxi g \in
    \cH_\eta\}.
  \end{aligned}
\end{equation}

% \begin{lemma}
%   Let $f\in \Dom(A_q)$. Then, $f(x) \to 0$ as $x\to \pm \infty$.
% \end{lemma}
% \begin{proof}
%   Since $\la \alpha_qf,1\ra<\infty$ and since $q'/q$ is bounded, we have that
% \end{proof}

\begin{proposition}
  \label{prop:AB*}
  For $f\in \Dom\Axiet$ and $g\in \Dom\Betxi$, we have
  \begin{equation}
    \label{eq:ABadj}
     \la \Axiet f,g\ra_\xi = \la f,\Betxi g\ra_\eta . 
  \end{equation}
\end{proposition}
\begin{proof}
  By \eqref{eq:albefadj2}, we have
  \[ \heta f \hxi g\Big|^t_{-t} \to \la \alxiet f, g\ra_\eta-\la
    f,\beetxi g\ra_\xi \quad \text{as } t\to \infty,\] with
  $\heta,\hxi$ as per \eqref{eq:heta}. Observe that
  \[ (\heta f)^2 = f^2 W_\eta,\quad (\hxi g)^2 = g^2 W_\xi.\]
  Hence, 
  $\heta f, \; \hxi g\in \rL^2(\R)$. Therefore, the above limit must
  be zero.
\end{proof}

\begin{proposition}
  \label{prop:AB0*}
  We have $\betxi^*\subseteq -\Axiet$ and $\axiet^*\subseteq -\Betxi$.
\end{proposition}
\begin{proof}
  Let $f\in \Dom\betxi^*\subset \cH_\eta$.  By definition, there exists
  an $h\in \cH_\xi$ such that
  \[ \la f, \beetxi g \ra_\eta = \la h, g\ra_\xi,\quad \text{for all }
    g\in \cD_0.\] Set
  \[ \tf(x) := \xi \int_{0}^x h\eta\xi^{-2},\quad x\in
    \R.\]
  Hence, by \eqref{eq:alphadef}, we have
  $\alxiet[\tf] = h$  almost everywhere.
  By \eqref{eq:albefadj2}, it then follows that
  \[ \int_\R \tf\, (\beetxi\bg)\, W_\eta =\la h ,g\ra_\xi ,\quad \text{for all
    } g\in \cD_0.\] Hence,
  \[ \int_\R (f-\tf)(\beetxi\bg) W_\eta = 0 \text{ for all } g\in
    \cD_0.\] Now, $f-\tf\in \rL^2(I,W_\eta)$ for every compact
  interval $I\subset \R$.  Hence, by Proposition \ref{prop:qperp}, $f$
  and $\tf$ differ by a multiple of $\xi$ a.e. on $I$. Since this is
  true for every $I$, it follows that $f-\tf$ is a multiple of $\xi$
  a.e. on all of $\R$. Therefore, $f\in \Dom\Axiet$ with
  $\alxiet f = h$.

  Next, let $f\in \Dom\axiet^*$.  By definition, there exists an
  $h\in \cH_\eta$ such that
  \[ \la f, \alxiet g \ra_\xi = \la h, g\ra_\eta,\quad \text{for all }
    g\in \cD_\xi\subset \cH_\eta.\] If $g=\xi$, then the above
  expressions vanish. Hence, $h\in\Ann_\eta\xi$, which implies that
  \[ \tf(x) := \teta(x) \int_{-\infty}^x h\xi W_\eta,\quad
    x\in \R\] is well defined.  By construction,
  $\beetxi\tf = h$ a.e.  It follows that
  \[ \int_\R \tf (\alxiet\bg)\, W_\xi = \la h, g\ra_\eta,\quad
    \text{for all } g\in \cD_\xi.\] Hence,
  \[ \int_\R (f-\tf) (\alxiet\bg)\,W_\xi = 0 \text{ for all } g\in
    \cD_\xi.\] Therefore, by Proposition \ref{prop:RanA0}, $f=\tf$ a.e. on
  $\R$.
\end{proof}

\begin{proposition}
  \label{prop:ABadj}
  We have $\Axiet=-\Betxi^*$ and $-\Betxi=\Axiet^*$.
\end{proposition}
\begin{proof}
  By Proposition \ref{prop:AB*}, we have $\Axiet\subseteq -\Betxi^*$
  and $-\Betxi \subseteq \Axiet^*$.  By Proposition \ref{prop:AB0*},
  \[ -\Betxi^* \subseteq -\betxi^* \subseteq \Axiet,\quad\text{and}
    \quad \Axiet^*\subseteq  
    \axiet^* \subseteq -\Betxi,\] as was to be shown.
\end{proof}

Let $T_\eta:\cH_\eta\to \cH_\eta$ denote the densely defined operator
with action $\taueta$ and with maximal domain
\begin{equation}
  \label{eq:domTeta}
  \Dom(T_\eta) = \{ f\in \ACloc^1(\R) \colon \taueta f \in \cH_\eta
  \} 
\end{equation}
Let $T_\xi:\cH_\xi\to \cH_\xi$ be the analogous operator with action
$\tauxi$.

\begin{proposition}
  \label{prop:TBA}
  Let $\eta,\xi \in \cR$ be such that
  \begin{equation}
    \label{eq:chietaxi}
    \chiof{\eta,\xi} = \eival_0 \eta\xi,\quad \eival_0\in \R.
  \end{equation}
  Then,
  \begin{equation}
    \label{eq:Tetaxi}
    \begin{aligned}
      T_\eta &= \Betxi\Axiet+\eival_0,\\
      T_\xi &= \Axiet \Betxi+\eival_0+2
    \end{aligned}
  \end{equation}
 in the the sense of naturally
  defined composition; i.e.
  \begin{align*}
     \Dom \Betxi\Axiet &= \{ f\in \Dom\Axiet \colon \Axiet f \in
                         \Dom\Betxi \};\\ 
     \Dom \Axiet\Betxi &= \{ f\in \Dom\Betxi \colon \Betxi f \in \Dom\Axiet \}.
  \end{align*}  
\end{proposition}

\begin{lemma}
  \label{lem:xietaeigen}
  Relation \eqref{eq:chietaxi} is equivalent to either of the following:
  \begin{align}
    \label{eq:tauetaxi}
    \taueta \xi &= \eival_0 \xi,\\
    \label{eq:tauxiteta}
    \tauxi\teta &= (\eival_0+2)  \teta.
  \end{align}
\end{lemma}
% \noindent 
% In other words, $\xi$ is a formal eigenfunction of $\taueta$ if and
% only if $\teta$ is a formal eigenfunction of $\tauxi$.
\begin{proof}
  The equivalence of \eqref{eq:chietaxi} and \eqref{eq:tauetaxi}
  follows directly from the definition \eqref{eq:taudef}.  The
  following differential relation can be verified by direct
  calculation:
  \begin{equation}
      \label{eq:chi1}
      \chiof{\xi,\teta}= e^{x^2} \chiof{\eta,\xi}+2 \xi\teta
  \end{equation}
  Hence,
  \[ \teta \taueta\xi = e^{x^2} \chiof{\eta,\xi} =
    \chiof{\xi,\teta}-2 \xi\teta = \xi (\tauxi-2)\teta . \] This
  proves the equivalence of \eqref{eq:tauetaxi} and
  \eqref{eq:tauxiteta}.
%    The equivalence of    \eqref{eq:tauetaxi} and
%    \eqref{eq:tauxiteta} follows immediately. 
\end{proof}

\begin{lemma}
  \label{lem:taufac}
  Suppose that \eqref{eq:chietaxi} holds.  Then,
  \begin{align}
      \label{eq:taueta1}
    \taueta &= \beetxi  \alxiet + \eival_0,\\
      \label{eq:tauxi1}
    \tauxi &=  \alxiet\beetxi  +\eival_0+2
  \end{align}
\end{lemma}
\begin{proof}
  By definition and direct calculation,
  \begin{equation}
  \label{eq:chi2}
  \chiof{\eta,f} \xi-
  \chiof{\eta,\xi}f  =  \teta^{-1} \WRof{ \eta\teta,  \WRof{\xi,f}}
  \end{equation}
  Hence, by \eqref{eq:alphadef} - \eqref{eq:taudef},
  \begin{align*}
    \taueta f
    &= (\txi\eta^2)^{-1}\WRof{ \eta\teta,\WRof{\xi,f}}+
      \xi^{-1} f      \taueta\xi\\ 
    &= \txi^{-1}\WRof{\teta,\eta^{-1}\WRof{\xi,f}}+ \eival_0 f\\
    &= \beetxi\alxiet f + \eival_0 f
  \end{align*}
  By definition and direct calculation,
  \[
    \begin{aligned}
    \alpha_{\teta,\xi}f&= e^{x^2}  \beetxi f,\\
    \beta_{\xi,\teta}f
    &= e^{-2x^2}\eta^{-1} \WRof{e^{x^2}\xi,f}= \eta^{-1}
    \WRof{\xi,e^{-x^2} f} = \alxiet(e^{-x^2}f)
    \end{aligned}
    \]
  Applying \eqref{eq:taueta1} with $\eta\mapsto \xi,\;\xi\mapsto
  \teta,\; \eival_0 \mapsto \eival_0+2$, and using
  \eqref{eq:tauxiteta} gives
  \begin{align*}
    \tauxi
    &= \beta_{\xi,\teta}  \alpha_{\teta,\xi} + \eival_0+2 =
      \alxiet \beetxi + \eival_0 +2.
  \end{align*}
\end{proof}

\begin{proof}
  By Lemma \ref{lem:taufac}, the formal factorization relations
  \eqref{eq:taueta1} and \eqref{eq:tauxi1} hold.  The differential
  expressions $\taueta, \tauxi$ are limit point at $\pm \infty$; see
  \cite[Theorem 2.4]{CL}.  Hence, $T_\eta, T_\xi$ are self-adjoint.
  By Proposition \ref{prop:ABadj}, $\Axiet, -\Betxi$ are
  adjoint. Hence, the natural compositions $\Axiet \Betxi$ and
  $\Betxi \Axiet$ are self-adjoint by a standard theorem in functional
  analysis; see for example \cite[Theorem 13.13]{Rudin}.  Since there
  are no boundary conditions at $\pm \infty$ the LHS and the RHS of
  \eqref{eq:Tetaxi} must coincide.
\end{proof}
\begin{proposition}
  \label{prop:Anorm}
  Let $\eta,\xi\in \cR$ such that \eqref{eq:chietaxi} holds.  Let
  $\psi\in \cH_\eta$ be an eigenfunction of $T_\eta$; i.e.
  $\taueta\psi = \eival \psi,\; \eival\in \R$.
  Then, necessarily   $\eival_0>\eival$ and 
  \[ \hpsi := \alxiet\psi \] is an eigenfunction of $T_\xi$ with
  eigenvalue $\eival+2$.  Moreover,
  \begin{equation}
    \label{eq:hpsinorm}
    \Vert \hpsi \Vert^2_\xi = (\eival_0-\eival) \Vert \psi \Vert_\eta^2
  \end{equation}
\end{proposition}

\begin{proof}
  By assumption $\psi \in \Dom T_\eta$.  Hence, 
  by Proposition \ref{prop:TBA}, $\psi \in \Dom \Axiet$, and 
  $\hpsi\in H_\xi$.  
  By \eqref{eq:ABadj} and Proposition \ref{prop:TBA}, we have
  \begin{align*}
    \la \hpsi,\hpsi \ra_\xi
    &= -\la \Betxi \Axiet \psi,\psi \ra_\eta\\
    &= \la (\eival_0 - T_\eta) \psi,\psi \ra_\eta =
      (\eival_0-\eival) \la \psi,\psi\ra_\eta.
  \end{align*}
  By construction,
  \begin{align*}
    \beetxi \hpsi
    &= \beetxi \alxiet \psi  = (\eival-\eival_0) \psi\in    \cH_\eta,\\
    \alxiet\beetxi \hpsi
    &= (\eival-\eival_0) \hpsi \in \cH_\xi.
  \end{align*}
  Hence, $\hpsi \in \Dom \Betxi$ and $\beetxi\hpsi \in \Dom \Axiet$.
  It follows that, $\hpsi \in \Dom T_\xi$, with
  \[ T_\xi \hpsi = (\Axiet \Betxi+ \eival_0+2) \hpsi = (\eival+2) \hpsi.\]
\end{proof}

\begin{theorem}
  \label{thm:Tetxi}
  Let $\eta,\xi\in \cR$ be such that \eqref{eq:chietaxi} holds.  Let
  $\sigma(T_\eta), \sigma(T_\xi)$ denote the spectral sets of the
  indicated operators.  Then,   $\eival_0= \max \sigma(T_\eta)$.  Moreover,
  \[ \sigma(T_\xi-2) = \sigma(T_\eta)\setminus \{ \eival_0 \}.\]
\end{theorem}
\begin{proof}
  Both $T_\eta$ and $T_\xi$ have a pure point spectrum; for a proof
  see Section 5.9 of \cite{Titchmarsh} or Problem 1 in Ch. 9 of
  \cite{CL}.  By assumption, $\psi_0:=\xi$ has no real zeros.
  Relation \eqref{eq:chietaxi} is equivalent to
  \[ \taueta \psi_0 = \eival_0 \psi_0.\]
  Hence, $\eival_0  = \max \sigma(T_\eta)$
  by the Sturm oscillation theorem.   Let $\psi_n\in \cH_\eta,\; n\in
  \N$ be  the other eigenfunctions of $T_\eta$ with
  \[ \taueta\psi_n =   \eival_n \psi_n,\quad n=1,2,\ldots \]
  and $\eival_0>\eival_1>\eival_2>\cdots$.  Set
  \[ \hpsi_n = \alxiet \psi_n,\quad n=1,2,\ldots.\] By Proposition
  \ref{prop:Anorm}, these are all eigenfunctions of $T_\xi$ with
  eigenvalues $\eival_n,\; n=1,2,\ldots$.  By Proposition
  \ref{prop:RanA0}, $\Ran \Axiet = \cH_\xi$.  Since $\{ \psi_n
  \}_{n\in \Nz}$ is complete in $\cH_\eta$, it follows that
  $\{ \hpsi_n \}_{n\in \N}$ is complete in $\cH_\xi$.  Therefore
  $\sigma(T_\xi) = \{ \eival_n \}_{n\in \N}$.
\end{proof}
Deift \cite{deift} proved a more general version of Theorem
\ref{thm:Tetxi} regarding the spectra of general self-adjoint
operators related by a factorization/Darboux transformation.
Nonetheless, for the sake of comprehensiveness we prefer to state and
prove a more restrictive version limited to Sturm-Liouville operators.
The above proof of Theorem \ref{thm:Tetxi} is more accessible than
Deift's more abstract argument.

\section{The formal theory of Exceptional Hermite polynomials}

Classical Hermite polynomials are orthogonal polynomials defined by
the recurrence relation
\begin{equation}
  \label{eq:herm3trr}
  H_0=1,\quad x H_n = \frac12 H_{n+1} + n H_{n-1},\quad n=1,2,\ldots
\end{equation}
They satisfy
 the following orthogonality relation:
\begin{equation}
  \label{eq:hortho}
  \int_{-\infty}^\infty H_m(x) H_n(x) e^{-x^2} dx = \sqrt{\pi}\, 2^n
  n! \delta_{n,m}
\end{equation}
Hermite polynomials are known as classical orthogonal polynomials,
because they also arise as solutions of the Hermite differential
equation
\begin{equation}
  \label{eq:hermde}
  y '' - 2x y' = \eival y
\end{equation}
with $y=H_n,\eival = -2n,\; n\in \Nz$.  Note that multiplication of
\eqref{eq:hermde} by $-e^{-x^2}$ gives a singular Sturm-Liouville
problem on $(-\infty,\infty)$, namely
\[ - (e^{-x^2} y)' = -\eival e^{-x^2}y.\] This SLP is limit-point at
$\pm \infty$, so no explicit boundary conditions are required.

% The Hermite polynomials may also be defined in terms of the Rodrigues
% formula
% \begin{equation}
%   \label{eq:hermrr}
%   h_n(x) = (-1)^n    e^{x^2} \partial_x^n e^{-x^2},\; n=0,1,2,\ldots
% \end{equation}

Exceptional Hermite polynomials are a generalization of the classical
Hermite polynomials because they satisfy a second-order, Hermite-like
differential equation.  Each family of such polynomials is indexed by
a partition, and so we begin by recalling some relevant definitions.

\begin{definition}
  A partition $\lambda$ of a natural number $N$ is a non-increasing,
  finitely supported sequence of non-negative integers
  $\lambda_1\geq \lambda_2\geq \ldots \ge 0$ that sum to $N$.  The
  length $\ell$ of a partition $\lambda$ is defined to be the smallest
  $\ell\in \Nz$ such that $\lambda_{\ell+1}=0$.  Thus,
  $N=\lambda_1 + \cdots + \lambda_\ell$, with $\lambda_\ell>0$.
\end{definition}

% \begin{definition}
%   Let $\lambda,N,\ell$ be as above.  For $k\ge \ell$, we will call
%   \begin{equation}
%     \label{eq:Mlamdef}
%     \Mlamk := \{ \lambda_i + \ell-i \colon i=1,\ldots, k \}
%   \end{equation}
%   the $k$-index set of $\lambda$.
% \end{definition}

Fix a partition $\lambda$, and let
\begin{align}
  \label{eq:Clamdef}
  \Clamof{l}
  &:= 2^{l(l-1)/2}\prod_{1\le i<j\le l}
    (\lambda_i-\lambda_j+j-i)\\
  \label{eq:pilam}
  \pilamof{l}(n)
  &:= \prod_{i=1}^l(n-N-\lambda_i+i),\quad k\in \N,\\
  \label{eq:midef}
  m_i &:= \lambda_i + \ell-i,\quad i\in \N;\\
  \label{eq:etalam}
  \etalam
  &:= (\Clam_\ell)^{-1}\WRof{H_{m_\ell},\dots,H_{m_1}},\\
  \label{eq:Ilamdef}
  \Ilam
  &:= \{ n\in \Nz \colon n-N+\ell \ge 0 \text { and }
    \pilamof{\ell}(n)\neq 0.\},  \\
  \label{eq:Hlamn}
  \Hlamof n
  &:=(2^\ell\Clam_\ell\pilamof{\ell}(n))^{-1}
    \WRof{H_{m_\ell},\dots,H_{m_1},H_{n-N+\ell}},\quad      n\in \Ilam
\end{align}
where $\WR$ denotes the usual Wronskian determinant.

Observe that $n-N+\ell\in \{ m_1, \ldots, m_\ell \}$ if and only if
$\pilamof{\ell}(n) =0$. Thus, $\Hlamof{n},\; n\in \Ilam$ is a non-zero
Wronskian of classical Hermite polynomials.  Also note that if $N=0$
and $\lambda$ is the empty partition, then $\etalam=1$ and
$\Hlamof{n} = H_n$, is the classical $n$th degree Hermite polynomial.

\begin{proposition}
  \label{prop:degHlam}
  Let $\etalam, \Hlamof{n},\; n\in \Ilam$ be as above. Then
  \begin{align}
    \label{eq:etadeg}
    \etalam(x) &= 2^N x^N + \text{ lower degree terms.}\\
    \label{eq:Hlamndeg}
    \Hlamof{n}(x) &= 2^n x^n + \text{ lower degree terms.}
  \end{align}
\end{proposition}
\begin{proof}
  Let $p_i,\; i=1,\ldots, l$ be a polynomial of degree $d_i$ with
  leading coefficient $c_i$.  Suppose that $d_1,\ldots, d_l$ are
  distinct and let $P = \WRof{p_1,\ldots, p_l}$.  Then
  \[ \deg P = \sum_{i=1}^l d_i - \frac12 l(l-1) =
    \sum_{i=1}^l (d_i-l+i).\]
  Consequently,
  \begin{align*}
    \deg \etalam 
    &= \sum_{i=1}^\ell (m_i-\ell+i)= \sum_i \lambda_i = N,  \\
    \deg \Hlamof n 
    &= n-N+\ell - (\ell+1) +\ell+1 + \sum_{i=1}^\ell (m_i-\ell-1+i)  = n.  
  \end{align*}

  The leading coefficient of $P$ is given by
  $\prod_{i=1}^k c_i \prod_{1\le i<j\le k} (d_j-d_i)$.  It 
  follows that $\Clamof{\ell}$ and $2^\ell\Clamof{\ell} \pilamof{\ell}(n)$
  are precisely the leading coefficients of $\etalam$ and $\Hlamof{n}$,
  respectively.
\end{proof}

The formulation of $\etalam$ and $\Hlamof{n}$ in \eqref{eq:etalam} and
\eqref{eq:Hlamn} is based on Wronskians of size $\ell$ and $\ell+1$
respectively.  However, it is possible to define these polynomials
using Wronskians of any size $l\ge \ell$.
\begin{proposition}
  \label{prop:shiftprop}
  Let $\lambda$ be a partition of $N$ with length $\ell$. Fix $l\ge
  \ell$, and set
  \[ m_{i,l} = \lambda_i+l - i,\quad i=1,\ldots, l.\]
  Then,
  \begin{align}
    \label{eq:etalaml}
    \etalam
    &=      \Clamof{l}^{-1}\WRof{H_{m_{l,l}},\ldots, H_{m_{1,l}}},\\
    \label{eq:Hlamnl}
    \Hlamof{n}
    &=(2^l\Clam_l\pilamof{l}(n))^{-1}
    \WRof{H_{m_{l,l}},\dots,H_{m_{l,1}},H_{n-N+l}}
  \end{align}
%   We have
%   \begin{align}
%     \label{eq:hHlamHlam}
%     \hHlam &= \Clam_{\ell,N}  \Hlam,\\
%     \label{eq:hHlamHlam}
%     \hHlamsub{n}
%     &= \Clam_{\ell,N}  \frac{\Hlamof n}{\pilam_\ell(n)} ,\quad n\in \Ilam,
%   \end{align}
%   where
%   \begin{align}
%     \label{eq:pilam}
%      \Clam_{a,b} &:= 2^{\frac12(a-b)(a+b-1)}\prod_{j=a+1}^b \prod_{i=1}^{j-1}
%     (\lambda_i-\lambda_j+j-i)
% %    \pilam(n) &:= \prod_{i=1}^\ell (n-k_i)
%   \end{align}
\end{proposition}
\begin{proof}
  The proof is by  induction on $l-\ell$.  If $l=\ell$, there is
  nothing to prove.  Suppose that \eqref{eq:etalaml} holds for a
  particular $l \ge \ell$. Observe that
  \[ m_{i,l+1} = m_{i,l}+1,\quad i=1,\ldots, l,\] with $m_{l+1,l+1}=0$.
  Since
  \[ H_n' = 2n H_{n-1},\quad n=1,2,\ldots, \]
  we have
  \begin{align*}
    \WR(H_{m_{l+1,l+1}}, H_{m_{l+1,l}},\ldots, H_{m_{l+1,1}})
    &=     \WR(1, H_{m_{l,l}+1},\ldots, H_{m_{l,1}+1})\\
    &=     \WR(H_{m_{l,l}+1}',\ldots, H_{m_{l,1}+1}')\\
    &=     2^l \prod_i (m_{l,i}+1)\WR(H_{m_{l,l}},\ldots, H_{m_{l,1}}).
  \end{align*}
  Observe that
  \[ \sum_{i=1}^{l+1} m_{l+1,i} = N + l(l+1)/2,\]
  the leading term coefficient of the
  left-side Wronskian is $2^N     \Clamof{l+1}$
  , it follows that
  \[ \WR(H_{m_{l+1,l+1}}, H_{m_{l+1,l}},\ldots, H_{m_{l+1,1}}) = \Clamof{l+1}
    \etalam.\]
  Relation \eqref{eq:Hlamnl} is proved in an anolgous fashion.
\end{proof}
\begin{corollary}
  \label{cor:hratfunc}
  For every $l\ge \ell$ and  $n \in \Ilam$, we have
  \begin{equation}
    \label{eq:linvar}
    \frac{\Hlamof{n}}{\etalam} =    \frac{    \WRof{H_{m_{l,l}},\ldots,
        H_{m_{1,l}},H_{n-N+l}}}{ 2^l
      \pilamof{l}(n)\WRof{H_{m_{l,l}},\ldots, H_{m_{1,l}}} }   . 
  \end{equation}
\end{corollary}
In other words, the rational function $\etalam^{-1}\Hlamof{n}$ can
be defined without invoking the normalizing constant $\Clamof{l}$ of
\eqref{eq:Clamdef}.

A particularly interesting case of this formulation occurs when we
take $l=N$.  By \eqref{eq:Hlamndeg}, the set $\Ilam$ is simply the
degree sequence of the polynomial family $\{\Hlamof{n}\}_n$.
Let
\begin{align}
  \label{eq:kidef}
  k_i &:= m_{i,N} =  \lambda_i + N - i,\quad i=1,\ldots, N,\\
  \Klam &:= \{ k_1,\ldots, k_N \};
\end{align}
and observe that
\begin{equation}
  \label{eq:Klam}
    \Klam=\{ 0,1,\ldots, N-\ell-1 \} \cup  \{ m_i + N-\ell \colon
    i=1,\ldots, \ell\}.  
\end{equation}
Recall that $n-N+\ell\in \{ m_{1}, \ldots, m_{\ell} \}$ if and only if
$\pilamof{\ell}(n) =0$.  It follows that $\Klam = \Nz \setminus \Ilam$
is the set of exceptional degrees missing from $\Ilam$, and that there
are precisely $N=\deg \etalam$ such ``exceptional'' degrees.
Moreover, by \eqref{eq:etalaml} and \eqref{eq:Hlamnl}, we have
\begin{equation}
  \label{eq:Wrki}
\begin{aligned}
 \etalam &\propto \WRof{H_{k_N},\ldots, H_{k_1}} ,\\
   \Hlamof{n}&\propto \WRof{H_{k_N},\ldots, H_{k_1},n} ,\quad n\in \Ilam.
\end{aligned}
\end{equation}
Thus both $\etalam$ and the exceptional polynomials $\Hlamof{n}$ can
be formulated as Wronskians involving the exceptional degrees.  Of
course, for most partitions $N>\ell$ and so the formulations in
\eqref{eq:etalam} and \eqref{eq:Hlamn} are more economical.  See
\cite{GGM18} for a characterization of the smallest determinant that
can be used to represent a Wronskian of Hermite polynomials.

% Next, define
% \begin{align}
%     \label{eq:hHlam}
%     \hHlam &= \WRof{H_{k_N},\ldots, H_{k_1}}\\
%     \label{eq:pilam}
%     \pilam_b(n) &= \prod_{i=1}^b (n-k_i)\\
%     \label{eq:hHlamn}
%     \hHlamsub{n} &= \pilam_N(n)^{-1}\WRof{H_{k_N},\ldots, H_{k_1}, H_n}
%                ,\quad n\in \Ilam.
% \end{align}
% As we now show, the polynomials $\hHlamsub{n}$ are simply exceptional
% Hermite polynomials with a different choice normalization.  These
% normalized polynomials have certain technical advantages that will
% become evident in the sequel.
Just like their classical counterparts, exceptional Hermite
polynomials satisfy a second-order eigenvalue relation.  Set
\begin{equation}
  \label{eq:taulamdef}
  \taulam y :=   y''-2\lp x+\frac{\etalam'}{\etalam}\rp y'+ \lp \frac{\etalam''}{\etalam}+ 2 x
  \frac{\etalam'}{\etalam}\rp  y;
\end{equation}
i.e., $\taulam$ is the second-order differential expression in the
left-side of \eqref{eq:tauform} with $\eta=\etalam$.  

\begin{proposition}
  \label{prop:chiWr}
  Let $m_1,\ldots, m_l, m\in \Nz$ be distinct non-zero integers.  Set
  \begin{align*}
    \eta &= \WRof{H_{m_l},\ldots, H_{m_1}} \\
    \xi &= \WRof{H_{m_l},\ldots, H_{m_1},H_m} 
  \end{align*}
  Then,
  \begin{equation}
    \label{eq:chietxi}
    \chi(\eta,\xi) = 2(l-m) \eta \xi.
  \end{equation}
\end{proposition}
\begin{proof}
  Relation \eqref{eq:chietxi} is equivalent to
 \begin{equation}
    \label{eq:DEetxi}
    (\tau_\eta+2(m-l)) \xi =0
\end{equation}
  Set
  \[ \ttau_\eta y = y'' -2x y' + 2 (\log \eta)''y.\]
  An elementary calculation shows that
  \[ \tau_\eta (\eta y) =  \eta \ttau_\eta y.\]
  Thus, it suffices to show that
  \begin{equation}
    \label{eq:ttaueigen}
     (\ttau_\eta+2(m-l)) (\eta^{-1} \xi) = 0.
  \end{equation}
  Let
  \[ \tauz y = y''-2x y' \] be the classic Hermite differential
  expression.    Define the $l$th order differential expression
  \[ \kappa y =\frac{ \WRof{H_{m_l},\ldots, H_{m_1}, y} }{
      \WRof{H_{m_l},\ldots, H_{m_1}} }= y^{(l)}-(\log\eta)'
    y^{(l-1)} + \cdots ,\]
  and observe that
  \[    \kappa H_m = \eta^{-1}\xi. \]
 Since
  \[ \tauz H_m = -2m H_m,\]
  it  suffices to show that
  \[ (\ttaulam-2l) \kappa = \kappa \tauz,\] This, in turn, is equivalent to the
  relation
  \[ [\tauz,\kappa] = (2l-2(\log\eta)'' )\kappa.\] By
  construction, $\ker\kappa = \lspan \{H_{m_1},\ldots, H_{m_l}\}$.
  Since these are all eigenfunctions of $\tauz$, the commutator
  $[\tauz,\kappa]$ annihilates $\ker \kappa$.  By inspection,
  $[\tauz, \kappa]$ has order $l$. Hence, $[\tauz,\kappa] =f \kappa$
  for some function $f(x)$.  Observe that
  \[ [\tauz, \kappa]y = (2l-2(\log\eta)'') y^{(l)} +\cdots\]
  Hence, $f=2l-2(\log\eta)''$.  This proves
  \eqref{eq:ttaueigen}, which is equivalent to \eqref{eq:chietxi}.
\end{proof}
\begin{corollary}
  \label{cor:xeigen}
  Let $\lambda$ be a partition and $n\in \Ilam$ an allowed degree. Then,
  \begin{equation}
    \label{eq:chieigen}
    \chiof{\etalam, \Hlamof{n}} = 2(N-n)\etalam    \Hlamof{n}. 
  \end{equation} 
\end{corollary}
\begin{proof}
  Since $\chi$ is bilinear, this follows by \eqref{eq:Wrki} and the
  preceding Proposition.
\end{proof}

% We are interested in those polynomials $\eta$ for which the
% eigenfunctions of the Sturm-Liouville problem \eqref{eqWet:aSL} are
% all polynomials.

The following converse was
proved in \cite{GGM14}.
\begin{proposition}
  Let $\eta(x)$ be a non-zero polynomial. Let
  $\{ y_n  \colon \deg y = n \}_{n\in I}$ be the family of
  polynomial solutions of
  \[ \chi(\eta, y) = \eival \eta y\quad \eival \in \C. \] Suppose that
  $\Nz\setminus I$ is finite; i.e., the family of polynomial solutions
  is missing at most finitely many degrees.  Then, there exists a
  partition $\lambda$ such that $I= \Ilam$, and such that, up to a
  multiplicative constant, $\eta = \etalam$ and
  $y_n = \Hlamof n,\; n\in \Ilam$.
\end{proposition}

\subsection{Multi-step  factorization chains.}
In this section, we show that every exceptional operator $T_\lambda$
is connected to the classical
\[ \tauz y = y''-2x y' \] by a finite factorization chain of
exceptional Hermite operators.  We will describe two such chains: one
will connect an exceptional Hermite operator $\taulam$ to $\tauz$ and
the other will connect $\tauz$ to $\taulam$.

Let $\lambda$ be a partition of $N$. Let $k_1,\ldots, k_N\in \Klam$ be
the exceptional degrees as per \eqref{eq:kidef}, and let
$n_1,\ldots, n_l\in \Ilam$ be a list of allowed degrees.
% Set $\lami{0} = \lambda$, and let $\lami{i},\; i=1,\ldots, p$ be the
% partition whose exceptional degrees are
% \[ K_{\lami{i}} = \Klam \cup \{ n_1,\ldots, n_i \},\quad i=1,\ldots,
%   p.\]
Set
\[ \eta_i = \WRof{ H_{k_N},\ldots,H_{k_1},H_{n_1},\ldots, H_{n_i}},\quad
i=0,1,\ldots, l,\] and introduce the differential expressions
\begin{align*}
  \alpha_i
  &= \alpha_{\eta_{i+1},\eta_i},\quad i=0,\ldots,l-1 ,\\
  \beta_i
  &= \alpha_{\eta_i,\eta_{i+1}}\\
  \tau_i &= \tau_{\eta_i},\quad i=0,\ldots, l.
\end{align*}
In particular, $\tau_0 = \taulam$.
\begin{proposition}
  \label{prop:fachain}
  With the above definitions, we have
  \begin{align}
    \label{eq:chietaii+1}
    \chiof{\eta_i,\eta_{i+1}}
    &= 2(N+i-n_{i+1})\eta_i\eta_{i+1},\quad
    i=0,\ldots, l-1,\\
    \label{eq:fachain}
    \tau_i &= \beta_i \alpha_i + 2(N+i-n_{i+1}),\quad i=0,\ldots, l-1;\\
    \label{eq:fachain+1}
    \tau_{i+1} &= \alpha_i \beta_i + 2(N+i+1-n_{i+1}).
  \end{align}
\end{proposition}
\begin{proof}
  Relation \eqref{eq:chietaii+1} follows by Corollary
  \ref{cor:xeigen}. Relations \eqref{eq:fachain} and
  \eqref{eq:fachain+1} follow by Lemma \ref{lem:taufac}.
\end{proof}

Next, we describe a factorization chain that connects $\tau_{\lambda}$
to $\tauz$ using $\lambda_1$ steps.  Let
$k_1,k_2,\ldots, k_N\in \Klam$ be the exceptional degrees, as per
\eqref{eq:kidef}.  Since $k_1=\lambda_1+N-1$, every $n\ge \lambda_1+N$
is an allowed degree; i.e. an element of $\Ilam$.

\begin{definition}
  We call the smallest $\lambda_1$ elements of $\Ilam$ the
  \emph{sporadic} degrees of $\Ilam$.
\end{definition}
\begin{proposition}
  \label{prop:spdeg}
  An $n\in \Ilam$ is a sporadic degree if and only if there exists an
  exceptional degree $k\in \Klam$ such that $n<k$.
\end{proposition}
\begin{proof}
  By \eqref{eq:kidef}, the $N$ exceptional degrees are contained in
  $\{ 0,1,\ldots, \lambda_1+N-1\}$, a set of cardinality
  $\lambda_1+N$. Hence, $n\in \Ilam$ is dominated by some exceptional
  degree if and only if $n< \lambda_1+N-1$.  By the preceding remark,
  there are precisely $\lambda_1$ such degrees.
\end{proof}

\begin{proposition}
  \label{prop:lamj}
  Let $\lambda$ be a partition of $N$, and let
  $n_1<n_2<\ldots$ be the allowed degrees in $\Ilam$ 
  listed in increasing order.  For $j\in \Nz$, set
  \begin{align}
    \label{eq:lamjdef}
    \lami{j}_i &= \max(\lambda_i-j,0)\quad i\in \N,
  \end{align}
  and let $\ell_j$ be the length of $\lami{j}$.  Then,
  \begin{equation}
    \label{eq:elljnj}
     \ell_j = N-n_{j+1}+j,\quad j\in \Nz.
  \end{equation}
\end{proposition}
\begin{proof}
  Observe that if
  $j\geq \lambda_1$, then $\lami{j}= \emptyset$ is the trivial
  partition.  Thus it suffices to establish  \eqref{eq:elljnj} for
  $j=0,\ldots, \lambda_1-1$.
  
  By definition, $i\leq \ell_{j}$ if and only if $\lambda_i-j\ge 1$.
  In other words, the sequence $(\ell_0,\ell_1,\ell_2,\ldots)$ and $\lambda$
  are dual partitions.  Observe  that if $\lambda_i-j \ge 1$, then
  $i-\ell_j\le 0$.  Conversely, $\lambda_i-j\le 0$ if and only if
  $i-\ell_j\ge 1$.  Hence,
  \[ \lambda_i-i \neq j-\ell_j,\quad i\in \N,\; j\in \Nz.\] Using
  $\sqcup$ for disjoint union, it follows that
  \[ \{\lambda_i-i\colon i=1,\ldots, N \} \sqcup \{ j-\ell_j \colon
    j=0,\ldots, \lambda_1-1 \} = \{ -N,-N+1,\ldots, \lambda_1-1 \}.\]
  Therefore, \[ \{ j-\ell_j+N \colon j=0,\ldots, \lambda_1-1\}\] is
  precisely the set of sporadic degrees in $\Ilam$. Relation
  \eqref{eq:elljnj} follows.
\end{proof}

\begin{proposition}
  \label{prop:safachain}
  Let $\lambda$ be a partition, with $\lami{j},\; j\in \Nz$ the
  sequence of partitions defined by \eqref{eq:lamjdef}.  Define the
  differential expressions
  \begin{align*}
    \alpha_j &= \alpha_{\eta_{\lami{j+1}},\eta_{\lami{j}}},\quad j=0,\ldots,
    \lambda_1-1,\\
    \beta_j &= \alpha_{\eta_{\lami{j}},\eta_{\lami{j+1}}};\\
    \tau_j &= \tau_{\eta_{\lami{j}}},\quad j=0,\ldots, \lambda_1.
  \end{align*}
  Then,
  \begin{align}
    \label{eq:fachain1}
    \tau_j &= \beta_j \alpha_j + 2\ell_j,\quad j=0,\ldots,\lambda_1-1;\\
    \tau_{j+1} &= \alpha_j \beta_j + 2(\ell_j+1).
  \end{align}
\end{proposition}
\begin{proof}
  % Let $k_1,\ldots,k_N$ be the exceptional degrees and
  Let $n_1<\ldots< n_{\lambda_1}$ be the sporadic degrees of $\lambda$
  listed in increasing order.  Let
  \begin{align*}
     K_j &= \Klam \cup \{ n_1,\ldots, n_j \},\quad j=0,\ldots,
    \lambda_1,\\
    \eta_j &= \WRof{k_N,\ldots, k_1,n_1,\ldots, n_j}.
  \end{align*}
  % By \eqref{eq:Klam}, we have  $N_1=n_1 = N-\ell $.  Hence,
  % \[ \Klam \cup \{ n_1 \} = \{ 0,\ldots, N-\ell \} \cup \{
  %   k_\ell,\ldots, k_1\} = \{ \lami{1}_i + N+1-i \colon i=1,\ldots, N+1
  %   \}.\]
  % Continuing inductively, for $j=1,\ldots, \lambda_1$, we have
  % \[ \Klam \cup \{ n_1,\ldots, n_j \} = \{ \lami{j}_i + N+j-i \colon
  %   i=1,\ldots, N+j \}.\] Hence, by Proposition \ref{prop:shiftprop},
  % we have
  % \[ \eta_j \propto \eta_{\lami{j}},\quad j=0,\ldots, \lambda_1.\]
  % By construction,
  % \begin{align*}
  %   N_{j+1} &= N_j - \ell_j,\quad j=0,\ldots,  \lambda_1-1,\\
  %   n_{j+1} &= N_j-\ell_j+n_j+1,    \quad j=1,\ldots, \lambda_1-1. 
  % \end{align*}
  By \eqref{eq:elljnj} and \eqref{eq:chietaii+1}, 
  \[ \chiof{\eta_j,\eta_{j+1}} = 2(N+j-n_{j+1})\eta_j\eta_{j+1} =
    2\ell_j \eta_j \eta_{j+1}.\] The desired conclusion now follows by
  Proposition \ref{prop:fachain}.
\end{proof}

For example, consider the partition $\lambda = (3,3,1,1,0,\ldots)$ of
$N=8$.  The exceptional degrees are
\[ \Klam = \{ 0,1,2,3,5,6,9,10 \}.\]
The degree set is therefore,
\[ \Ilam = \{ 4,7,8,11,12,13,\ldots \},\]
with $4,7,8$ being the sporadic degrees.  In terms of the above terminology,
\begin{align*}
  K_1 &= \{ 0,1,2,3,4,5,6,9,10 \}, &\lami{1} &= (
                                               2,2,0,0,0,\ldots),\quad
  N_1 = 4,\; \ell_1 = 2\\
  K_2 &= \{ 0,1,2,3,4,5,6,7,9,10 \}, &\lami{2} &= ( 1,1,0,0,0,\ldots),
  \quad N_2 = 2,\; \ell_2 = 2\\
  K_3 &= \{ 0,1,2,3,4,5,6,7,8,9,10 \}, &\lami{3}  &=(0,0,0,0,0,\ldots),
\quad N_3= 0,\; \ell_3 = 0.                                                   
\end{align*}

\subsection{The norm identity.}
In this section, we present and prove a certain algebraic identity
that will allow us to derive a formula for the norming constants of
the exceptional Hermite polynomials.

\begin{proposition}
  Let $m_1,\ldots, m_l,m\in \Nz$ be distinct non-negative integers.
  Set
  \begin{align*}
    \eta_0 &= 1\\
    \xi_{0}&= H_m\\
    \eta_i &=  \WRof{H_{m_1},\ldots, H_{m_i}},\quad i=1,\ldots, l;\\
    \xi_{i} &=  \WRof{H_{m_1},\ldots, H_{m_i},H_m}
  \end{align*}
  Also, let $\rho_0 =0$, and recursively define
  \[ \rho_{i+1} = 
    \frac{ \xi_{i} \xi_{i+1}}{\eta_i\eta_{i+1}}
    + 2 (m-m_{i+1})\rho_{i},\quad i=0,\ldots, l-1.
  \]
  We then have,
  \begin{equation}
    \label{eq:normid}
     \lp \frac{\xi_l}{\eta_l}\rp^2e^{-x^2} - 2^{l} \prod_{i=1}^l
    (m-m_i) \xi_0^2 e^{-x^2} = \lp \rho_{l}\, e^{-x^2}\rp
    '.
  \end{equation}
\end{proposition}

\begin{proof}
  By \eqref{eq:chietxi}, we have
  \[ \chi(\eta_i,\eta_{i+1}) = 2(i-m_{i+1}) \eta_i \eta_{i+1},\quad
    i=0,\ldots, l-1.\] Hence, by the  argument used in
  Proposition \ref{prop:fachain}, the differential expressions
  \begin{align*}
    \alpha_i &= \alpha_{\eta_{i+1},\eta_i},\quad i=0,\ldots, l-1\\
    \beta_i &= \beta_{\eta_{i},\eta_{i+1}},\quad i=0,\ldots, l-1\\
    \tau_i &= \tau_{\eta_i},\quad i=0,\ldots, l.
  \end{align*}
  constitute the factorization chain
  \begin{align*}
    \tau_i &= \beta_i \alpha_i + 2(i-m_{i+1}),\quad i=0,\ldots, l-1\\
    \tau_{i+1} &= \alpha_i \beta_i + 2(i+1-m_{i+1}).
  \end{align*}
  Next, observe that
  \begin{align*}
    \alpha_i \xi_i
    &= \frac{\WRof{\WRof{H_{m_1},\ldots, H_{m_i},H_{m_{i+1}}},
      \WRof{H_{m_1},\ldots, H_{m_i},H_m}}}{
      \WRof{H_{m_1},\ldots, H_{m_i}}}\\
    &=    \WRof{H_{m_1},\ldots, H_{m_i},H_{m_{i+1}}, H_m} = \xi_{i+1}.
  \end{align*}
  Again, by \eqref{eq:chietxi}, we have
  \[ \tau_i(\xi_i) = 2(i-m)  \xi_i,\quad i=0,\ldots,    l.\]
  Hence, by \eqref{eq:albefadj}, we have
  \begin{align*}
    \frac{\xi_{i+1}^2}{\eta_{i+1}^2}
    e^{-x^2}-2(m-m_{i+1})\frac{\xi_i^2}{\eta_i^2} e^{-x^2} 
    &= \lp\frac{\xi_i \xi_{i+1}}{\eta_i \eta_{i+1}} e^{-x^2}
      \rp',\quad i=0,\ldots, l-1.
  \end{align*}
  Hence, by induction,
  \begin{align*}
    \frac{\xi_{l}^2}{\eta_{l}^2}e^{-x^2} - 2^l \prod_{i=1}^{l}(m-m_i)
    \xi_0^2 e^{-x^2} =\lp e^{-x^2} \sum_{i=1}^l
    \prod_{j=1+i}^l (2m-2m_j) \frac{\xi_j \xi_{j-1}}{\eta_j \eta_{j-1}} \rp'
  \end{align*}
  Again, by induction,
  \[\rho_k =
     \sum_{i=1}^k
    \prod_{j=1+i}^l (2m-2m_j) \frac{\xi_j \xi_{j-1}}{\eta_j
      \eta_{j-1}},\quad k=1,\ldots, \ell-1. \]
  The desired identity \eqref{eq:normid} follows immediately.
\end{proof}
\section{The $\rL^2$ theory}

The following theorem was proved by Krein for Sturm-Liouville problems
on the half-line, and independently by Adler for Sturm-Liouville
problems on a bounded interval. The Adler argument can be extended
without difficulty to the case of an infinite interval \cite{GFGU15}.

\begin{definition}
  We say that $\lambda$ is an even partition 
  $\lambda_{2i-1} = \lambda_{2i}$ for every $i\in \N$.
\end{definition}

\begin{theorem}[Krein-Adler]
  The polynomial $\etalam(x)$ has no real zeros if and only if
  $\lambda$ is an even partition.
\end{theorem}

Henceforth we assume that $\lambda$ is an even partition.  In that
case,
\[\Wlam := W_{\etalam} =\etalam^{-2}e^{-x^2}\]
is a non-singular weight on $(-\infty,\infty)$ with finite moments.
Let $\cH_\lambda = \rL^2(\R,W_\lambda)$ denote the corresponding Hilbert
space with weighted inner product,
\[ \la f,g\ra_\lambda = \int_\R f \bg W_\lambda = \int_R
  \frac{f}{\etalam} \frac{\bg}{\etalam} e^{-x^2} ,\quad f,g\in \cH_\lambda.\]
% For a given partition $\lambda$ of length $\ell$, define the
% polynomial
% \begin{equation}
%   \label{eq:pilambdadef}
%   \pi_\lambda(m) = \prod_{i=1}^\ell (m-m_i),
% \end{equation}
% with $m_i,\; i=1,2,\ldots$ as defined in \eqref{eq:midef}.
% Observe that, by  \eqref{eq:Klam}, for $n\in \Z$ we have
% $\pi_\lambda(n)>0$ if and only if $n\in \Ilam$.
\begin{proposition}
  \label{prop:XHorthog}
  Let $\lambda$ be an even partition. Then, the corresponding exceptional
  Hermite polynomials enjoy the following orthogonality property:
  \begin{equation}
    \label{eq:XHorthog}
    \la \Hlamof n, \Hlamof {n'}\ra_\lambda =
\sqrt{\pi}\, 2^{n-N} \frac{n!}{\pilamof{N}(n)}\,     \delta_{n,n'} ,\quad
    n,n'\in \Ilam,
  \end{equation}
  with $\pilamof{N}(n)$ as defined in \eqref{eq:pilam}.
\end{proposition}
\begin{proof}
  Being the eigenfunctions of a SLP \eqref{eq:WetaSL}, the
  $\Hlamof{n},\; n\in \Ilam$ are orthogonal with respect to $\Wlam =
  W_{\etalam}$. It remains to establish the form of the norming
  constants in \eqref{eq:XHorthog}.
  By Corollary \ref{cor:hratfunc} and \eqref{eq:normid},
  \begin{align*}
    \la \Hlamof{n}, \Hlamof{n} \ra_\lambda
    &=    \int_\R
      \lp\frac{\Hlamof{n}(x)}{\etalam(x)}\rp^2 e^{-x^2} dx,\quad n\in \Ilam\\
    &=
      \int_\R
      \lp \frac{\WRof{H_{k_N},\ldots,    H_{k_1},H_{n}}}{
      2^N \pilamof{N}(n)\WRof{H_{k_N},\ldots, H_{k_1}} }  \rp^2
      e^{-x^2} dx,    \\
    &= \frac{1}{2^N \pilamof{N}(n)} \int_R H_n(x)^2 e^{-x^2} dx\\
    &= \sqrt{\pi} 2^{n-N}\frac{n!}{2^N \pilamof{N}(n)}  .
    \end{align*}
\end{proof}

\begin{theorem}
  \label{thm:main}
  Let $\lambda$ be an even partition. Then the SLP of form
  \eqref{eq:Weta} with $\eta = \etalam$ has eigenvalues
  \[ \eival_n = 2(n-N),\quad n\in \Ilam \]
  The corresponding eigenfunctions are $\Hlamof{n},\; n\in \Ilam$,
  while the norming constants are given by \eqref{eq:XHorthog}.
\end{theorem}
\begin{lemma}
  Let $\lambda$ be a partition and let $\lami{j},\; j\in \N$ be the
  partitions defined in \eqref{eq:lamjdef}.    If
  $\lambda$ is an even partition, then so is every $\lami{j}$.
\end{lemma}
\begin{proof}
  Suppose that $\lambda_{2i-1} = \lambda_{2i}$ for
  all $i \in \N$. It follows that
  \[ \lami{j}_{2i-1} = \max(\lambda_{2i-1} - j,0) =
    \max(\lambda_{2i}-j,0) = \lami{j}_{2i}\,\quad i,j\in \N.\]
\end{proof}
\begin{proof}[of Theorem \ref{thm:main}]
  Multiplication of \eqref{eq:WetaSL} by $-\Wlam^{-1}$ and changing
  $\eival \to -\eival$, gives \eqref{eq:tauform}, which may be
  expressed as
  \[ \taulam y =  \eival y.\]
  By \eqref{eq:taudef} and  \eqref{eq:chieigen}, we have
  \[ \taulam \Hlamof{n} = 2(N-n) \Hlamof{n},\quad n\in \Ilam.\] Hence,
  $2(n-N)$ is an eigenvalue of \eqref{eq:Weta} with corresponding
  eigenfunction $\Hlamof{n},\;n\in \Ilam$.

  It remains to show that there are no either eigenvalues; i.e., that
  $\{\Hlamof{n} \}_{n\in \Ilam}$ is a complete orthogonal basis.  To
  show this, we employ the factorization chain described in Proposition
  \ref{prop:safachain}.  Let $\lami{j},\; j=0,\ldots, \lambda_1$ be
  the partitions defined in \eqref{eq:lamjdef}, and let
  $\eta_j = \eta_{\lami{j}}$.  By the above Lemma, these are all even
  partitions, and hence each $\eta_j \in \cR$.

  Let $T_j = T_{\eta_j},\; j=0,1,\ldots,\lambda_1$ be a sequence of
  self-adjoint operators with action $\tau_{\eta_j}$ and domains as
  per \eqref{eq:domTeta}.  By Proposition \ref{prop:TBA}, the formal
  factorization chain of Proposition \ref{prop:safachain} corresponds
  to a factorization chain of operators
  \begin{align*}
    T_j &= B_j A_j + 2\ell_j,\quad j=0,\ldots,\lambda_1-1,\\
    T_{j+1} &= A_j B_j + 2\ell_j+2,\quad j=0,\ldots,\lambda_1-1,\\
  \end{align*}
  where $\ell_j$ is the length of partition $\lami{j}$, and where
  $A_j, B_j$ are first order operators with action
  $\alpha_{\eta_{j+1},\eta_j}, \beta_{\eta_{j+1},\eta_j}$,
  respectively, and domains as given in \eqref{eq:domAdomB}.  Hence,
  $2\ell_j$ is the largest eigenvalue of
  $T_j,\; j=0,\ldots, \lambda_1$.  Hence, by Theorem \ref{thm:Tetxi},
  \[ \sigma(T_{j+1}-2) = \sigma(T_j)\setminus \{ 2\ell_j\} . \]
  It follows that
  \[ \sigma(T_0) = \{ 2\ell_j - 2 j \colon j=0,\ldots, \lambda_1-1\}
    \cup \sigma(T_{\lambda_1}).\]
  However, since $\tau_{\lambda_1} = \tauz$ is the classical Hermite
  operator, we have by \eqref{eq:elljnj} that
  \begin{align*}
    \sigma(T_0)
    &= \{ 2\ell_j - 2 j \colon j=0,\ldots, \lambda_1-1\}
      \cup (-2\Nz)\\
    &= \{ 2(N-n_{j+1}) \colon j=0,\ldots, \lambda_1-1     \}
      \cup (-2\Nz)\\
    &= \{ 2(N-n) \colon n\in \Ilam \}.
  \end{align*}
\end{proof}

\begin{corollary}
  Let $K = \{ k_1,\ldots, k_N \} \subset \Nz$ be a list of $N$
  non-negative integers such that
  \[ \sum_{i=1}^N k_i = \frac12 N(N+1).\] Set
  \[ \lambda_i = k_i+i-N,\; i\in \N.\] Then, 
  $\{ \Hlamof{n} \}_{n\in \Ilam}$ is a family of exceptional Hermite
  polynomials such that $K=\Klam$; i.e. the exceptional degrees are
  precisely $K$.  Moreover, if $\lambda$ is an  even partition, then
  the corresponding   $\{ \Hlamof{n} \}_{n\in \Ilam}$ are
  eigenfunctions of the SLP   \eqref{eq:WetaSL} whose spectrum, up to
  a shift with $N$, differs from the spectrum of the classical Hermite
  SLP by the removal of eigenvalues at position $2k,\; k\in K$.
\end{corollary}

% \begin{acknowledgement}
% This research has been financed in part by the Spanish MICINN under grants PGC2018-096504-B-C33 and RTI2018-100754-B-I00 and the European Union under the 2014-2020 ERDF Operational Programme and by the Department of Economy, Knowledge, Business and University of the Regional Government of Andalusia (project FEDER-UCA18-108393).
% \end{acknowledgement}

\end{document}